\documentclass{amsart}

\usepackage[dvips]{graphicx}
\usepackage{amssymb}

\newtheorem{theorem}{Theorem}[section]
\newtheorem{lemma}[theorem]{Lemma}
\newtheorem{cor}[theorem]{Corollary}

\theoremstyle{definition}

\newtheorem{example}[theorem]{Example}

\newtheorem{rem}[theorem]{Remark}

\theoremstyle{remark}

\numberwithin{equation}{section}

\newcommand{\R}{{\mathbb R}}
\newcommand{\Q}{{\mathbb Q}}
\newcommand{\N}{{\mathbb N}}
\newcommand{\Z}{{\mathbb Z}}
\newcommand{\sgn}{{\rm sgn}}
\newcommand{\stab}{{\rm stab}}

\begin{document}

\title[Quotients of index two and general quotients]{Quotients of index two and general quotients in a space of orderings}

\thanks{This research was supported through the programme ``Research in Pairs'' by the Mathematisches Forschungsinstitut Oberwolfach in 2011. An early version of the paper has been published an Oberwolfach Preprint (OWP
2011-36).}

\thanks{The second author was supported by the NSERC of Canada.}

\thanks{This research was initiated when the second author was visiting University of Silesia in Katowice, Poland.}

\author{Pawe\l \ G\l adki and Murray Marshall}

\address{Institute of Mathematics,
University of Silesia, \newline \indent
ul. Bankowa 14, 40-007 Katowice, Poland}
\email{pawel.gladki@us.edu.pl}

\address{Department of Mathematics and Statistics,
University of Saskatchewan, \newline \indent
106 Wiggins Rd., Saskatoon, SK S7N 5E6, Canada}
\email{marshall@math.usask.ca}

%\date{October 26, 2010}

\begin{abstract} In this paper we investigate quotient structures and quotient spaces of a space of orderings arising from subgroups of index two. We provide necessary and sufficient conditions for a quotient structure to be a quotient space that, among other things, depend on the stability index of the given space. The case of the space of orderings of the field $\Q(x)$ is particularly interesting, since then the theory developed simplifies significantly. A part of the theory firstly developed for quotients of index 2 generalizes in an elegant way to quotients of index $2^n$ for arbitrary finite $n$. Numerous examples are provided.
\end{abstract}

\maketitle

\section{Introduction and notation}

The theory of abstract spaces of orderings was developed by Murray Marshall in a series of papers from the late 1970s, and provides an abstract framework for studying orderings of fields and the reduced theory of quadratic forms in general. The monograph \cite{M3} will be of frequent use here as far as background, notation, and main results are concerned. Spaces of orderings also occur in a natural way in other more general settings: from considering maximal orderings on semi-local rings, orderings on skew fields, or orderings on ternary fields. The axioms for spaces of orderings have been also generalized in various directions -- to quaternionic schemes, to spaces of signatures of higher level, or to abstract real spectra that are used to study orderings on commutative rings.

An elegant first-order description of spaces of orderings is given in \cite{Marshall2006}. Here, we prefer to use the earlier description given in \cite{M3}, i.e., a space of orderings is a pair $(X, G)$ such that $X$ is a nonempty set, $G$ is a subgroup of $\{1, -1\}^X$, which contains the constant function $-1$, separates points of $X$, and satisfies two additional axioms.
Considering $X$ as a subset of the character group $\chi(G)$ (here by characters we mean group homomorphisms $x: G \rightarrow \{-1, 1\}$) via the natural embedding $X \hookrightarrow \chi(G)$ obtained by identifying $x \in X$ with the character $G \ni a \mapsto a(x) \in \{-1, 1\}$, and denoting,
for any pair $a, b \in G$:
$$D(a,b) = \{c \in G: \forall x \in X (c(x) = a(x) \vee c(x) = b(x))\},$$
the two additional axioms state
\begin{enumerate}
\item[(1)] if $x \in \chi(G)$ satisfies $x(-1) = -1$, and if
$$\forall a, b \in \ker (x) \ (D(a,b) \subseteq \ker(x)),$$
then $x$ is in the image of the natural embedding $X \hookrightarrow \chi(G)$, and
\item[(2)] $\forall a_1, a_2, a_3, b, c \in G \ \exists d \in G \ [(b \in D(a_1, c) \wedge c \in D(a_2, a_3)) \Rightarrow (b \in D(d, a_3) \wedge d \in D(a_1, a_2))]$.
\end{enumerate}
Spaces of orderings are easily made into a category by introducing morphisms in the following way: a morphism $F$ from a space of orderings $(X_1, G_1)$ to a space of orderings $(X_2, G_2)$ is a function $F: X_1 \rightarrow X_2$ such that
$$\forall b \in G_2 \ (b \circ F \in G_1).$$

In this paper we shall investigate quotient objects in the category of spaces of orderings. They have been first studied in \cite{M-1} and, up to present day, remain rather mysterious creatures: if $(X, G)$ is a space of orderings, and $G_0$ is a subgroup of $G$ containing the element $-1$, we denote by $X_0$ the set $X|_{G_0}$ of all characters from $X$ restricted to $G_0$, and call the pair $(X_0, G_0)$ a quotient structure. In the case when $(X_0, G_0)$ is a space of orderings, we call it a quotient space of $(X, G)$. It has been shown in \cite{M-1} that quotient spaces are, indeed, quotients in the category of spaces of orderings. At the same time the problem of determining whether a given quotient structure is a quotient space proves to be surprisingly challenging. %In particular, checking whether the axiom (2) is satisfied is a very difficult task.
The main objective of this paper is to address this question in some special cases. In what follows we shall introduce some more notation, and then proceed to explain our motivation for the current work as well as present our main results.

A space of orderings $(X, G)$ has a natural topology introduced by the family of subbasic clopen Harrison sets:
$$H_X(a) = \{x \in X: a(x) = 1\},$$
for a given $a \in G$, which makes $X$ into a Boolean space (\cite[Theorem 2.1.5]{M3}). Whenever it is clear from the context which space of orderings we consider, we shall simply write $H(a)$ instead of $H_X(a)$.

For any multiplicative group $G$ of exponent 2 with distinguished element $-1$, we set $X = \{x \in \chi(G): x(-1) = -1\}$ and call the pair $(X,G)$ a fan. A fan is also a space of orderings (\cite[Theorem 3.1.1]{M3}).
We can also consider fans within a bigger space of orderings, and for this we need the notion of a subspace of a space $(X, G)$ -- a subset $Y \subseteq X$ is called a subspace of $(X, G)$, if $Y$ is expressible in the form $\bigcap_{a \in S} H_X(a)$, for some subset $S \subseteq G$. For any subspace $Y$ we will denote by $G|_Y$ the group of all restrictions $a|_Y$, $a \in G$. The pair $(Y, G|_Y)$ is a space of orderings itself (\cite[Theorem 2.4.3]{M3}, \cite[Theorem 2.2]{M-2}). Finally, if $(X, G)$ is a space of orderings, by a fan in $(X, G)$ we understand a subspace $V$ such that the space $(V, G|_{V})$ is a fan. One easily checks that any one- or two-element subset of a space of orderings forms a fan -- thus one- or two-element fans are called trivial fans.

The stability index $\stab (X, G)$ of a space of orderings $(X, G)$ is the maximum $n$ such that there exists a fan $V \subseteq X$ with $|V| = 2^n$ (or $\infty$ if there is no such $n$). It can be shown that the stability index of a space $(X, G)$ is at most equal to $k$ if every basic set in $X$ can be expressed as an intersection of $k$ Harrison sets (\cite[Theorem 3.4.2]{M3}, \cite[Theorem 6.2]{M}). Spaces of stability index 1 are also called spaces with the strong approximation property or SAP spaces.

We say that $(X, G)$ is the direct sum of the spaces of orderings $(X_i, G_i)$, $i \in \{1, \ldots, n\}$, denoted $(X, G) = \coprod_{i=1}^n (X_i, G_i) = (X_1, G_1) \sqcup \ldots \sqcup (X_n, G_n)$, if $X$ is the disjoint union of the sets $X_1, \ldots, X_n$, and $G$ consists of all functions $a: X \rightarrow \{-1, 1\}$ such that $a|_{X_i} \in G_i$, $i \in \{1, \ldots, n\}$. In this case $G = G_1 \oplus \ldots \oplus G_n$, with the role of the distinguished element $-1$ played by $(-1, -1, \ldots, -1)$. Further, we say that $(X, G)$ is a group extension of the space of orderings $(\overline{X}, \overline{G})$, if $G$ is a group of exponent 2, $\overline{G}$ is a subgroup of $G$, and $X = \{x \in \chi(G): x|_{\overline{G}} \in \overline{X}\}$. Since $G$ decomposes as $G = \overline{G} \times H$, we shall also write $(X,G) = (\overline{X}, \overline{G}) \times H$ to denote group extensions. Both direct sums and group extensions are spaces of orderings (\cite[Theorem 4.1.1]{M3}, \cite[Remark 2.8, Remark 3.7]{M-1}).

For $S\subseteq G$, $\langle S\rangle$ denotes the subgroup of $G$ generated by $S$. For $S \subseteq \chi(G)$, $S^{\perp}:= \{ g\in G : \sigma(g)=1 \forall \sigma \in S\}$ and $\langle S \rangle := \chi(G/S^{\perp})$, the closed subgroup of $\chi(G)$ generated by $S$.  %We shall use the notation $\langle a_1, \ldots, a_n \rangle$ to denote the group generated by the elements $a_1, \ldots, a_n$.

For a space of orderings $(X, G)$ we define the connectivity relation $\sim$ as follows: if $x_1, x_2 \in X$, then $x_1 \sim x_2$ if and only if either $x_1 = x_2$ or there exists a four element fan $V$ in $(X, G)$ such that $x_1, x_2 \in V$. The equivalence classes with respect to $\sim$ are called the connected components of $(X, G)$.
It is known that if  $(X, G)$ is a finite space of orders and $X_1, \ldots, X_n$ are its connected components, then $(X, G) = (X_1, G|_{X_1}) \sqcup \ldots \sqcup (X_n, G|_{X_n}).$ Moreover, the $(X_i, G|_{X_i})$, are either one element spaces or proper group extensions (\cite[Theorem 4.2.2]{M3}, \cite[Theorem 4.10]{M-2}).

For any space of orderings $(X, G)$, a quadratic form with entries in $G$ is an $n-$tuple $\phi = (a_1, \ldots, a_n)$, $a_1, \ldots, a_n \in G$; $n$ is called the dimension of $\phi$, and, for each $x \in X$, the signature of $\phi$ at $x$ is $\sgn_x \phi = \sum_{i=1}^n a_i(x) \in \Z$. %If $\phi = (a_1, \ldots, a_n)$, $\psi = (b_1, \ldots, b_m)$, $c \in G$, and $k \in N$, we define
\begin{eqnarray*}
\phi \oplus \psi & = & (a_1, \ldots, a_n, b_1, \ldots, b_m), \\
c \phi & = & (ca_1, \ldots, ca_n), \\
\phi \otimes \psi & = & a_1 \psi \oplus \ldots \oplus a_n \psi, \\
((a_1, \ldots, a_n)) & = & (1, a_1) \otimes \ldots \otimes (1, a_n), \\
k \times \phi & = & \underbrace{\phi \oplus \ldots \oplus \phi}_k.
\end{eqnarray*}
Two forms $\phi$ and $\psi$ are isometric, written $\phi \cong \psi$, if they are of the same dimension and signatures for all $x \in X$. Further, we say that $\phi$ and $\psi$ are Witt equivalent, denoted $\phi \sim \psi$, if there exist integers $k, l \geq 0$ such that
$$\phi \oplus k \times (1, -1) \cong \psi \oplus l \times (1, -1).$$ The sum $\oplus$ and the product $\otimes$ of quadratic forms induce binary operations on the set of equivalence classes of the relation $\sim$ making it into a commutative ring with $1$ that is denoted by $W(X, G)$ and called the Witt ring associated to the space of orderings $(X, G)$. The ideal of the ring $W(X, G)$ additively generated by the set $\{(1, a): a \in G\}$ is denoted by $I(X, G)$ and called the fundamental ideal of $W(X, G)$.

For a formally real field $k$ denote by $X_k$ the set of all orderings of $k$, and by $G_k$ the multiplicative group $k^*/(\Sigma k^2)^*$ of all classes of sums of squares of $k$. $G_k$ is naturally identified with a subgroup of $\{-1,1\}^{X_k}$ via the homomorphism
$$k^* \ni a \mapsto \overline{a} \in \{-1, 1\}^{X_k}, \mbox{ where } \overline{a}(\sigma) = \left\{ \begin{array}{ll} 1, & \mbox{ if } a \in \sigma,\\
-1, & \mbox{ if } a \notin \sigma, \end{array} \mbox{ for } \sigma \in X_k, \right.$$
whose kernel is the set $(\Sigma k^2)^*$ of all nonzero sums of squares of $k$, and $(X_k, G_k)$ is a space of orderings (\cite[Theorem 2.1.4]{M3}). For simplicity we shall denote by the same symbol $a$ both an element $a \in k^*$, a class of sums of squares $a \in k^*/(\Sigma k^2)^*$, and a function $a \in \{-1, 1\}^{X_k}$.

Since the invention of abstract spaces of orderings there has been a considerable interest in the question of when such a space is realized as a space of orderings of a field. It seems likely that spaces of orderings exist that are not so realized but, so far, no such examples are known. A possible way of proving that a space of orderings is not realized in such a way is to give an example of a form $\phi \in W(X, G)$ such that
\begin{center} $\forall \sigma \in X (\sgn_{\sigma} \phi \equiv 0 \mod 2^n)$ and $\phi \notin I^n(X, G)$.\end{center}
Here $W(X, G)$ denotes the Witt ring of the space of orderings $(X, G)$, and $I^n(X, G)$ denotes the $n$-th power of its fundamental ideal. The equivalence:
\begin{center} $\forall \sigma \in X (\sgn_{\sigma} \phi \equiv 0 \mod 2^n)$ if and only if $\phi \in I^n(X, G)$.\end{center}
is valid for any $n \in \N$, if $(X, G)$ is a space of orderings of a field. This was conjectured by Marshall and is also known as {\em Lam's Open Problem B} \cite{Lam77}. For $n \leq 2$ the
equivalence is easy to show, but for $n \geq 3$ the proof uses a deep result from \cite{OVV} and \cite{V}. A short explanation of how to derive the equivalence from that result can be found
in \cite{Marshall2000}, and a longer exposition on the theme in \cite{DM03}.

The problem also has an affirmative solution for all spaces of orderings of stability index no greater than 3 \cite[Theorem 6.2]{M}. It is, therefore, desirable to seek for examples of spaces of orderings of high stability indices. A possible way of finding such spaces is to investigate quotients of already known spaces: for a given space of orderings $(X, G)$ one can imagine some large subset of $X$ that is itself not a fan, but, by descending to an appropriate subgroup $G_0$ of $G$, can be forced to become a fan in the quotient space $(X|_{G_0}, G_0)$, thus increasing the stability index. The main problem, of course, is that in doing so one obtains a myriad of quotient structures with no way of knowing if they are actually quotient spaces.

We begin our discussion in Section 2 with the case when $G_0$ is a subgroup of $G$ of index two and provide a necessary and sufficient condition for a quotient structute of an SAP space $(X,G)$ to be a quotient space. We then refine our condition and show that this refined condition is both necessary and sufficient for spaces of stability index two that satisfy an extra technical condition. The refinement process can be continued and provides a sequence of necessary conditions -- unfortunately, other than in the two aforementioned cases and the finite case, we do not know when these conditions are sufficient.

In Section 3 we consider the special case when the space $(X, G)$ is the space of orderings of the field $\Q(x)$. The refined condition for spaces of stability index two is here also sufficient and, moreover, is also equivalent to saying that the quotient space is profinite. This generalizes results previously obtained in \cite{gj1}. Section 4 contains a handful of examples to illustrate the theory.

Finally, in Section 5, we generalize the necessary conditions to the case of subgroups of arbitrary (possibly infinite) index and discuss some of the instances when they are also sufficient. Surprisingly enough, the theory simplifies significantly in the case of the space of orderings of the field $\Q(x)$ and we show that a quotient structure of any finite index is a quotient space if and only if it is profinite.

\section{Quotients of index two}\label{general theory}

Let $(X,G)$ be a space of orderings and $(X_0, G_0)$ a quotient structure of $(X,G)$. We search for necessary and sufficient conditions on $G_0$ for $(X_0,G_0)$ to be a quotient space of $(X,G)$.

\begin{example} Suppose $G_0$ is a subgroup of $G$ and $-1 \in G_0$. Suppose $\sigma_0 : G_0 \rightarrow \{\pm 1\}$ is a character satisfying $\sigma_0(-1)=-1$. Let $\sigma : G \rightarrow \{\pm 1\}$ be an extension of $\sigma_0$ to a character on $G$. If $(X,G)$ is a fan then $\sigma \in X$ so $\sigma_0 \in X_0$. Thus if $(X,G)$ is a fan and $-1 \in G_0$ then $(X_0,G_0)$ is also a fan. In particular, $(X_0,G_0)$ is a quotient space of $(X,G)$, and any quotient structure of a fan is a quotient space that is a fan itself.
\end{example}

Assume now that $G_0$ is a subgroup of index $2$ in $G$, $-1 \in G_0$. Since $G_0$ has index $2$ and $-1 \in G_0$, $G_0$ is determined by a character on $G/\{\pm 1\}$, i.e., there exists a unique $\gamma \in \chi(G)$, $\gamma(-1)=1$ such that $G_0 = \ker(\gamma)$.

\begin{theorem} \label{basicstuff} (i) Suppose $(X_0,G_0)$ is a quotient space of $(X,G)$, $(Y,G/Y^{\perp})$ is a subspace of $(X,G)$, $\gamma \in \langle Y\rangle$ and $Y_0 = Y|_{G_0}$. Then $(Y_0,G_0/Y^{\perp})$ is a quotient space of $(Y,G/Y^{\perp})$. (ii) Suppose $(X,G)$ is a group extension of $(X',G')$ and $\gamma' = \gamma|_{G'}$. If $\gamma' = 1$, then $(X_0,G_0)$ is a quotient space of $(X,G)$. If $\gamma' \ne 1$, $(X_0,G_0)$ is a quotient space of $(X,G)$ iff $(X_0',G_0')$ is a quotient space of $(X',G')$. Here, $G_0':= \ker(\gamma')$, $X_0' := X'|_{G_0'}$.
\end{theorem}

\begin{proof} (i) $(Y_0,G_0/Y^{\perp})$ is a subspace of $(X_0,G_0)$ so (i) is clear. (ii) If $\gamma'=1$ then $(X_0,G_0)$ is a group extension of $(X',G')$. If $\gamma'\ne 1$, then $(X_0,G_0)$ is a group extension of $(X_0',G_0')$. In either case the result is clear, e.g., by \cite[Theorem 4.11 (2) and Theorem 4.1.3 (2)]{M3}.
\end{proof}

One needs to realize that the situation where $(X_0, G_0)$ is a space of orderings is rather special.

\begin{example} Consider the finite space $(X, G)$ of six orderings $\{\sigma_1, \ldots, \sigma_6\}$ and of stability index 1, so that $|G| = 2^6$. Let $\gamma = \sigma_1 \cdot \ldots \cdot \sigma_6$. The quotient structure $(X_0, G_0)$ consists of six orderings $\overline{\sigma_1}, \ldots, \overline{\sigma_6}$ with $\overline{\sigma_6} = \overline{\sigma_1} \cdot \ldots \cdot \overline{\sigma_5}$, and $|G_0|=2^5$. By the structure theorem for finite spaces of orderings \cite[Theorem 4.2.2]{M3}, \cite[Theorem 4.10]{M-2} $(X_0,G_0)$ is not a space of orderings (e.g., because it is not SAP but contains no four element fans). %if $(X_0, G_0)$ is a space of orderings, then it is the group extension of a direct sum of three singleton spaces, and, in particular, contains four element fans. But no product of three elements of $X_0$ belongs to $X_0$ -- a contradiction.
\end{example}

\begin{theorem}\label{main} A necessary condition for $(X_0,G_0)$ to be a quotient of $(X,G)$ is that $\gamma \in X^4$.
\end{theorem}

Here $X^k := \{ \prod_{i=1}^k \sigma_i : \sigma_i \in X, i=1,\dots,k\}.$ Since the $\sigma_i$ are not required to be distinct, $\{ 1 \} \subseteq X^2 \subseteq X^4$. If $(X,G)$ is a fan then $X^4 = X^2 = \{ \gamma \in \chi(G) : \gamma(-1) = 1\}$.

\begin{proof} Suppose first that the restriction map $r : X \rightarrow X_0$ is not injective, so there exist $\sigma,\tau \in X$, $\sigma \ne \tau$, $r(\sigma)=r(\tau)$. In this case, $G_0 = \ker(\sigma\tau)$, so $\gamma = \sigma\tau \in X^2$. Suppose next that $r$ is injective. Since $r$ is continuous and injective and $X$ is compact, $r$ is a homeomorphism. Fix $g\in G$, $g\notin G_0$, and define $\phi : X_0 \rightarrow \{\pm 1\}$ by $\phi(r(\sigma)) = \sigma(g)$. $\phi$ is well-defined and continuous.

\smallskip

Claim: $\phi$ is not in the image of $G_0$ under the natural injection $\hat{} : G_0 \rightarrow \operatorname{Cont}(X_0, \{\pm 1\})$. For suppose $\phi = \hat{h}$, $h\in G_0$. Then, for any $\sigma\in X$, $\sigma(h) = r(\sigma)(h) = \hat{h}(r(\sigma))= \phi(r(\sigma)) = \sigma(g)$, so $g=h \in G_0$, contradicting $g \notin G_0$.

\smallskip

%On the other hand, since $g \notin G_0$, $\phi$ is not in the image of $G_0$ under the natural injection $\hat{} : G_0 \rightarrow \operatorname{Cont}(X_0, \{\pm 1\})$.
Suppose now that $(X_0,G_0)$ is a space of orderings. By the claim and \cite[Theorem 3.2.2]{M3} there exists a 4-element fan $V$ in $X_0$ such that $\prod_{\alpha \in V} \phi(\alpha) \ne 1$. The character $\gamma' :=\prod_{\alpha \in V} r^{-1}(\alpha) \in X^4$ is $\ne 1$ (because $\gamma'(g)=\prod_{\alpha \in V} \phi(\alpha) \ne 1$) but the restriction of $\gamma'$ to $G_0$ is equal to 1 (because $V$ is a fan). Thus $\gamma = \gamma' \in X^4$, as required.
\end{proof}

\begin{theorem}\label{sap} If the space of orderings $(X,G)$ is  SAP, then the necessary condition in Theorem \ref{main} is also sufficient.
\end{theorem}

\begin{proof}
Suppose first that $\gamma \in X^2$, say $\gamma = \sigma_1\sigma_2$, $\sigma_1,\sigma_2 \in X$. Let $\phi : X_0 \rightarrow \{\pm 1\}$ be continuous. Then $\phi\circ r : X \rightarrow \{\pm 1\}$ is continuous. Since $(X,G)$ is SAP, the natural injection $\hat{} : G \hookrightarrow \operatorname{Cont}(X, \{ \pm 1\})$ is surjective, i.e., $\phi\circ r = \hat{g}$ for some $g\in G$. Then $\sigma_1(g) = \phi(r(\sigma_1)) = \phi(r(\sigma_2)) = \sigma_2(g)$, so $g\in G_0$. This implies that the natural injection $G_0 \hookrightarrow \operatorname{Cont}(X_0,\{\pm 1\})$ is an isomorphism. Suppose next that $\gamma \in X^4$, $\gamma \notin X^2$, say $\gamma = \prod_{i=1}^4 \sigma_i$, $\sigma_i\in X$, $i=1,\dots,4$. In this case one sees, by a similar argument, that the natural injection $G_0 \hookrightarrow \operatorname{Cont}(X_0,\{\pm 1\})$ identifies $G_0$ with $$\{ \phi \in \operatorname{Cont}(X_0,\{\pm 1\}) : \prod_{i=1}^4 \phi(r(\sigma_i))=1\}.$$ In either case, $(X_0,G_0)$ can be viewed as the space of global sections of a sheaf of spaces of orderings as defined in \cite[Chapter 8]{M1}, so $(X_0,G_0)$ is a space of orderings by Theorem \ref{sheaf}, more specifically, by Corollary \ref{sheafcor}.
Note: If $\gamma \in X^2$ all of the stalks are singleton spaces; if $\gamma \in X^4\backslash X^2$ one of the stalks is a 4-element fan and the rest are singleton spaces. %Note: The results in \cite{M1} are phrased in terms of reduced Witt rings, not spaces of orderings, but the two categories are equivalent.
\end{proof}

\begin{example} \label{rofx} Theorems \ref{main} and \ref{sap} provide a convenient description of quotients of index two of the space of orderings of the field $\R(x)$, or, more generally, of the space of orderings of any formally real function field of transcendence degree $1$ over a real closed field.  %$R(x)$, for any real closed field.
This is because spaces of orderings of this sort are SAP. %Recall also that orderings of the field $\R(x)$ are easily described geometrically. %Note that, since the orderings of the field $\R(x)$ can be easily described geometrically, Theorems \ref{main} and \ref{sap} can be spelled out as a certain ``positivity condition'' that has to be satisfied by elements of the group $G_0$. %Since the description of the orderings on $\R(x)$ has a nice geometric flavor, the condition of Theorem \ref{main} reads as the following ``positivity condition'': a quotient structure $(X_0, G_0)$ of the space of orderings $(X_{\R(x)}, G_{\R(x)})$ of index 2 is a quotient space if and only if $G_0$ consists of elements positive at an even number of a given quadruple of infinitesimals.
\end{example}

\begin{example} \label{ex2_6} The condition of Theorem \ref{sap} fails to be sufficient if the stability index of the space $(X, G)$ is greater than 1.
\begin{enumerate}
\item Consider the space $(X, G)$, where $$X = \{\sigma_1, \sigma_2, \sigma_3, \sigma_1\sigma_2\sigma_3, \sigma_4, \sigma_5, \sigma_6, \sigma_4\sigma_5\sigma_6\},$$ with $|G| = 2^6.$ This is the direct sum of two four element fans. Let $\gamma = \sigma_1 \sigma_4$. The quotient structure $(X_0, G_0)$ with $G_0 = \ker \gamma$ is not a quotient space.
\item Consider the space $(X, G)$, where $$X = \{\sigma_1, \sigma_2, \sigma_3, \sigma_1\sigma_2\sigma_3, \sigma_4, \sigma_5, \sigma_6\},$$ with $|G|=2^6.$ This is the direct sum of a four element fan and three singleton spaces. Let $\gamma = \sigma_1\sigma_4 \sigma_5 \sigma_6$. The quotient structure $(X_0, G_0)$ with $G_0 = \ker \gamma$ is not a quotient space.
\item Consider the space $(X, G)$, where $$X = \{\sigma_1, \sigma_2, \sigma_3, \sigma_4, \sigma_1\sigma_3\sigma_4, \sigma_2\sigma_3\sigma_4, \sigma_5, \sigma_6\},$$ with $|G|=2^6.$ This is the direct sum of a connected space of six elements and two singleton spaces. Let $\gamma = \sigma_1\sigma_2 \sigma_5 \sigma_6$. The quotient structure $(X_0, G_0)$ with $G_0 = \ker \gamma$ is not a quotient space.
\item Consider the space $(X, G)$, where $$X = \{\sigma_1, \sigma_2, \sigma_3, \sigma_4, \sigma_1\sigma_3\sigma_4, \sigma_2\sigma_3\sigma_4, \sigma_5, \sigma_6, \sigma_7, \sigma_8, \sigma_5\sigma_7\sigma_8, \sigma_6\sigma_7\sigma_8\},$$ with $|G|=2^8.$ This is the direct sum of two connected spaces, each consisting of six elements. Let $\gamma = \sigma_1\sigma_2 \sigma_5 \sigma_6$. The quotient structure $(X_0, G_0)$ with $G_0 = \ker \gamma$ is not a quotient space.
\end{enumerate}
Details of proofs are left to the reader. In each case one uses the structure theorem \cite[Theorem 4.2.2]{M3}, \cite[Theorem 4.10]{M-2} for the finite spaces of orderings $(X_0,G_0)$ and shows that the resulting quotient structure is constructed in a way contradicting the theorem.
\end{example}

To simplify things \it we assume from now on that the space of orderings $(X,G)$ contains no infinite fans. \rm This is the case, for example, if the stability index of $(X,G)$ is finite. Recall that, for $\delta \in \chi(G)$, $X_{\delta} := \{ \sigma \in X : \sigma\delta \in X\} = X\cap \delta X$. %The connected components of $(X,G)$ are defined in \cite{M0}.
Since $(X,G)$ has no infinite fans, every connected component of $(X,G)$ is either singleton or has the form $X_{\delta}$ for some $\delta \in \chi(G)$, $\delta \ne 1$, $|X_{\delta}| \ge 4$ \cite[Theorem 4.6.1]{M3},  \cite[Theorem 2.6]{M0}.

The requirement that $\gamma \in X^4$ can be substantially refined as follows:

\begin{theorem} \label{nec cond} A necessary condition for $(X_0,G_0)$ to be a quotient of $(X,G)$ is that $\gamma= \prod_{i=1}^k \sigma_i$, $\sigma_i \in X$, $k = 2$ or
$k=4$ and $\gamma \notin X^2$, and, in the case where not all $\sigma_i$ are in the same connected component of $(X,G)$ and the connected components of the $\sigma_i$ in $(X,G)$ are not all singleton, either $k=2$ and exactly one of the connected components of the $\sigma_i$ is not singleton, or $k=4$, $\gamma \notin X^2$ and, after reindexing suitably, the connected component of $\sigma_3$ and $\sigma_4$ is $X_{\sigma_3\sigma_4}$ and either the connected component of $\sigma_1$ and $\sigma_2$ is $X_{\sigma_1\sigma_2}$ or
the connected component of $\sigma_i$ is singleton for $i = 1,2$.
\end{theorem}

\begin{proof} Denote by $Y$ the union of the connected components of $(X,G)$ which meet the set $\{ \sigma_1,\dots,\sigma_k\}$. According to \cite[Theorem 3.6]{M2} $Y$, more precisely $(Y,G/\Delta)$ where $\Delta := Y^{\perp}$, is a subspace of $(X,G)$. Denote by $Y_0$ the set of restrictions of elements of $Y$ to $G_0$. If we assume that $(X_0,G_0)$ is a quotient space of $(X,G)$ then %$(Y_0,G_0/\Delta)$ is a subspace of the space of orderings $(X_0,G_0)$ so it is itself a space of orderings, i.e.,
$(Y_0,G_0/\Delta)$ is a quotient space of $(Y,G/\Delta)$, by Theorem \ref{basicstuff}. In this way, we are reduced to the case where $X=Y$, i.e., each connected component of $(X,G)$ meets the set $\{ \sigma_1,\dots,\sigma_k\}$.

Denote by $(Z_j,G/Z_j^{\perp})$, $j\in J$ the connected components of $(X,G)$. Each $Z_j$ is singleton or has the form $X_{\delta}$, $\delta \ne 1$, $|X_{\delta}|\ge 4$, and  $Z_j \cap \{ \sigma_1,\dots,\sigma_k\} \ne \emptyset$ for each $j$, so $|J|\le k$. By hypothesis, $2 \le |J|$ and not all $Z_j$ are singleton. According to \cite[Corollary 7.5]{M1}, $(X,G)$ is the direct sum of the $(Z_j,G/Z_j^{\perp})$, $j \in J$. In particular, $\chi(G) = \prod_{j\in J} \langle Z_j \rangle$ (direct product of groups), where $\langle Z_j \rangle$ is the closed subgroup of $\chi(G)$ generated by $Z_j$. Since $|J|\ge 2$ this implies in particular that $\gamma \notin \langle Z_j \rangle$ for each $j$.

The restriction map $r: X \rightarrow X_0$ is injective on each $Z_j$. This is clear if $Z_j$ is singleton. If $Z_j = X_{\delta}$, $\delta \ne 1$, $|X_{\delta}| \ge 4$, then injectivity follows from the fact that $\gamma \notin \langle Z_j \rangle$. We also see in this latter case that $r(Z_j) \subseteq (X_0)_{\delta_0}$, where $\delta_0$ denotes the restriction of $\delta$ to $G_0$, and $\delta_0 \ne 1$ (because $\gamma \notin \langle Z_j \rangle$). It follows, using \cite[Lemma 4.6]{M-2} repeatedly (see \cite[Remark 2.1]{M0}), that the space of orderings $(X_0,G_0)$ is connected and, moreover, that there exists $\mu_0 \in \chi(G_0)$, $\mu_0\ne 1$ such that $X_0 = (X_0)_{\mu_0}$. This implies in turn that $X = X_{\mu}\cup X_{\gamma\mu}$ where $\mu$ is some fixed extension of $\mu_0$ to a character on $G$.
Since $|X|\ge 5$ it follows that at least one of $|X_{\mu}|, |X_{\gamma\mu}|$ is $\ge 4$. Reindexing we can assume $|X_{\gamma\mu}|\ge 4$.

If $|X_{\mu}|$ is also $\ge 4$, then, since $X$ has at least $2$ connected components, $X_{\mu}\cap X_{\gamma\mu} = \emptyset$ and $X_{\mu}$ and $X_{\gamma\mu}$ are the connected components of $X$, so $\chi(G) = \langle X_{\mu}\rangle \times \langle X_{\gamma\mu}\rangle$. If $k=2$ then, after reindexing, $\sigma_1 \in X_{\mu}$, $\sigma_2 \in X_{\gamma\mu}$ and since the two decompositions $\gamma = (\mu)(\gamma\mu)$ and $\gamma = (\sigma_1)(\sigma_2)$ must be the same (because the product is direct), $\mu= \sigma_1$ and $\gamma\mu=\sigma_2$. Since $\sigma_1(-1)=-1$, $\mu(-1)=1$, this is not possible. If $k=4$, then, arguing as before with the two decompositions of $\gamma$, we see that, after reindexing suitably, $\sigma_1,\sigma_2 \in X_{\mu}$, $\sigma_3,\sigma_4 \in X_{\gamma\mu}$, $\mu= \sigma_1\sigma_2$, and $\gamma\mu=\sigma_3\sigma_4$.

This leaves the case $|X_{\mu}|=2$, $|X_{\gamma\mu}|\ge 4$. If $k=2$ then $X$ has two components, one singleton and one equal to $X_{\gamma\mu}$. Suppose now that $k=4$, $\gamma \notin X^2$. Reindexing, we can suppose $\sigma_3,\sigma_4 \in X_{\gamma\mu}$. There are two subcases: either $X_{\mu}\cap X_{\gamma\mu} = \emptyset$ or $X_{\mu}\cap X_{\gamma\mu} \ne \emptyset$. Suppose first that $X_{\mu}\cap X_{\gamma\mu}= \emptyset$. Then $X_{\mu} = \{\sigma_1,\sigma_2\}$, $\mu = \sigma_1\sigma_2$, $\gamma\mu = \sigma_3\sigma_4$. In this case the connected components are $\{ \sigma_1\}$, $\{ \sigma_2\}$ and $X_{\sigma_3\sigma_4}$. Suppose now that $X_{\mu}\cap X_{\gamma\mu} \ne \emptyset$. Reindexing we may assume $X_{\mu} = \{ \sigma_1,\sigma_1\mu\}$, $\sigma_1\mu \in X_{\gamma\mu}$. Then $(\sigma_1\mu)(\gamma\mu) = \sigma_1\gamma = \sigma_2\sigma_3\sigma_4 \in X$, contradicting $\gamma \notin X^2$. Thus this case cannot occur.
\end{proof}

We remark that the quotient structures appearing in Example \ref{ex2_6} are precisely those for which the conditions of Theorem \ref{nec cond} fail to be satisfied.

It is natural to wonder if the necessary conditions on Theorem \ref{nec cond} are sufficient when $(X,G)$ has stability index two. We are unable to prove this in general. We are however able to prove the following:

\begin{theorem} \label{beyond sap} If $(X,G)$ has stability index two and just finitely many non-singleton connected components, then the necessary conditions of Theorem \ref{nec cond} are sufficient.
\end{theorem}

\begin{proof} An application of Corollary \ref{sheafcor} allows us to reduce to the case where $(X,G)$ has just finitely many connected components. By assumption, $\gamma = \prod_{i=1}^k \sigma_i$, $\sigma_i \in X$, $k=2$ or $k=4$ and $\gamma \notin X^2$. One can reduce further to the case where each connected component has non-empty intersection with the set $\{ \sigma_1,\dots,\sigma_k\}$. If each connected component is singleton, then either $k=2$ and $(X_0,G_0)$ is a singleton space or $k=4$, $\gamma \notin X^2$, and $(X_0,G_0)$ is a $4$-element fan. Suppose $(X,G)$ has at least two connected components and at least one of these is not singleton. If $k=2$ then $(X,G)$ has exactly two connected components, one singleton, one non-singleton, and $(X_0,G_0)$ is isomorphic to the non-singleton component of $(X,G)$. If $k=4$, $\gamma \notin X^2$, then either $(X,G)$ has two connected components which, after reindexing, are $X_{\sigma_1\sigma_3}$ and $X_{\sigma_3\sigma_4}$, or three connected components which, after reindexing, are $\{ \sigma_1\}$, $\{ \sigma_2\}$ and $X_{\sigma_3\sigma_4}$. In either case, $(X_0,G_0)$ is a group extension by a group of order $2$ of the direct sum of the residue space of $X_{\sigma_1\sigma_2}$ associated to $\sigma_1\sigma_2$ and the residue space of $X_{\sigma_3\sigma_4}$ associated to $\sigma_3\sigma_4$. Note: In the case where $\{ \sigma_1\}$ and $\{ \sigma_2\}$ are connected components, $X_{\sigma_1\sigma_2} = \{\sigma_1,\sigma_2\}$ and the associated residue space is a singleton space. This leaves us with the case where $(X,G)$ has just one connected component. If $\gamma X = X$ then $(X_0,G_0)$ is the residue space of $(X,G)$ associated to $\gamma$. Suppose $\gamma X \ne X$. $(X,G)$ is a group extension of a SAP space of orderings $(X',G')$ by a cyclic group of order two. Let $\gamma'$ denote the restriction of $\gamma$ to $G'$. The pair $(X_0',G_0')$ associated to $\gamma'$ is a quotient of $(X',G')$, by Theorem \ref{sap}. $(X_0,G_0)$ is a group extension of $(X_0',G_0')$ by a cyclic group of order two.
\end{proof}

%The necessary conditions of Theorem \ref{nec cond} are not sufficient in general.

\begin{rem} When the stability index of $(X,G)$ is three or more there are additional necessary conditions: %which are obtained recursively:
Suppose $\sigma_1,\dots,\sigma_k$ all belong to the same connected component $(Z, H)$ of $(X,G)$, $\gamma Z \ne Z$, and $(Z,H)$ is a group extension of a space of orderings $(Z',H')$ by a cyclic group of order $2$. Let $\gamma' = \prod_{i=1}^k \sigma_i'$, where $\gamma'$ resp., $\sigma_i'$, denotes the restriction of $\gamma$, resp., $\sigma_i$, to $H'$.  By Theorem \ref{basicstuff}, if $(X_0,G_0)$ is a space of orderings then the associated quotient structure $(Z_0',H_0')$ of $(Z',H')$ is also a space of orderings. Replacing $(X,G)$ and $\gamma$ by $(Z',H')$ and $\gamma'$, additional necessary conditions are obtained recursively, in an obvious way. In particular, the conditions of Theorem \ref{nec cond} must hold for $(Z',H')$ and $\gamma'$.  It is not known if these recursively defined necessary conditions are sufficient. %They are sufficient in certain special cases, by Theorem \ref{sap} and Theorem \ref{beyond sap}.
It is easy to see that they are sufficient if the space of orderings $(X,G)$ is finite.
\end{rem}

\section{The space of orderings of $\mathbb{Q}(x)$}\label{Qx} %$(X_{\mathbb{Q}(x)}, G_{\mathbb{Q}(x)})$}

We consider the space of orderings of $\mathbb{Q}(x)$, the function field in a single variable $x$ over the field $\mathbb{Q}$ of rational numbers. This space of orderings is studied in \cite{dmm}, \cite{gj1} and \cite{gj2}. We will denote this space of orderings by $(X,G)$ for short, i.e., in this section, $$(X,G) := (X_{\mathbb{Q}(x)}, G_{\mathbb{Q}(x)}).$$
%We set up some notation.
For a real monic irreducible $p$ in the polynomial ring $\mathbb{Q}[x]$, set $n_p := $ the number of real roots of $p$  and set $X_p :=$ the set of elements of $X$ compatible with the discrete valuation $v_p$ of $\mathbb{Q}(x)$ associated to $p$, so $|X_p| = 2n_p$. Set $X_{\infty} :=$ the set of orderings compatible with the discrete valuation $v_{1/x}$, so $|X_{\infty}|=2$.  For a transcendental real number $r$, set $\sigma_r :=$ the archimedian ordering of $\mathbb{Q}(x)$ corresponding to the embedding $\mathbb{Q}(x) \hookrightarrow \mathbb{R}$ given by $x \mapsto r$. $X$ is the (disjoint) union of the sets $X_p$, $p$ running through the real monic irreducibles in $\mathbb{Q}[x]$, $X_{\infty}$, and $\{ \sigma_r\}$, $r$ running through the transcendental real numbers. The non-singleton connected components of $(X,G)$ are the $X_p$, $p$ a real monic irreducible of $\mathbb{Q}[x]$, $n_p\ge 2$.

Every monic irreducible $p$ of $\mathbb{Q}[x]$ which is not real is positive at every element of $X$, i.e., it is equal to $1$ in $G$. The set of elements $$\{ -1 \} \cup \{ p : p \text{ is a real monic irreducible in } \mathbb{Q}[x]\},$$ more precisely, the image of this set in $G$, forms a $\mathbb{Z}/2\mathbb{Z}$-basis for $G$, i.e., every element of $G$ is expressible uniquely as $$(-1)^{\delta_0}\prod_{i=1}^k p_i^{\delta_i},$$ $k\ge 0$, $p_1,\dots,p_k$ distinct real monic irreducibles in $\mathbb{Q}[x]$, $\delta_0,\dots,\delta_k \in \{ 0,1\}$.

Fix $A$, $B$ where $A$ is a finite set of real monic irreducible polynomials of $\mathbb{Q}[x]$ and $B$ is a finite set of transcendental real numbers. Set $$Y := (\bigcup_{p\in A} X_p) \cup X_{\infty}\cup \{ \sigma_r : r \in B \}.$$ Consider the set $r_1< \dots < r_m$ of real numbers consisting of the real roots of the various polynomials  $p \in A$ together with the elements of $B$. Clearly $m:= \sum_{p\in A} n_p+|B|$. Choose rational numbers $s_1,\dots, s_{m+1}$ such that $$-\infty < s_1 <r_1 <s_2<r_2< \dots <s_m<r_m <s_{m+1}<+\infty.$$ Set $H :=$ the subgroup of $G$ generated by $-1$, the elements $p\in A$, and the elements $x-s_i$, $i=1,\dots,m+1$.

\begin{lemma} \label{key} (i) $(Y, G/Y^{\perp})$ is a subspace of $(X,G)$. (ii) $(Y,G/Y^{\perp})$ is the direct sum of the subspaces $(X_p, G/X_p^{\perp})$, $p\in A$, $(X_{\infty}, G/X_{\infty}^{\perp})$, and $(\{ \sigma_r\}, G/\{\sigma_r\}^{\perp})$, $r\in B$. (iii) $(X|_H, H)$ is a quotient space of $(X,G)$. (iv) The spaces of orderings $(Y, G/Y^{\perp})$ and $(X|_H, H)$ are isomorphic via the natural maps $H \hookrightarrow G \rightarrow G/Y^{\perp}$, $Y \hookrightarrow X \rightarrow X|_H$.
\end{lemma}

\begin{proof} (i) and (ii) are consequences of \cite[Theorem 3.6]{M2} and \cite[Corollary 7.5]{M1}, respectively. One can also prove (ii) using the approximation theorem for $V$-topologies, e.g., see \cite{W}. (iii) is a consequence of (iv), so it suffices to prove (iv).

Claim 1: The map $Y \rightarrow X|_H$ is bijective. Let $S_0 :=$ the set of orderings satisfying $x<s_1$, $S_{m+1} :=$ the set of orderings satisfying $x>s_{m+1}$, and $S_i := $ the set of orderings satisfying $s_i<x<s_{i+1}$, $i=1,\dots,m$. Clearly $X = S_0\cup \dots \cup S_{m+1}$ (disjoint union) and $X|_H = S_0|_H\cup \dots \cup S_{m+1}|_H$ (disjoint union). Each $p\in A$ has constant sign on $S_0$. E.g., if $n_p$ is even (resp., odd) then $p$ is constantly positive (resp., constantly negative) on $S_0$. It follows that $S_0|_H$ is a singleton set. Also, $Y\cap S_0$ is a singleton set. A similar argument shows that $S_{m+1}|_H$ is a singleton set and $S_{m+1}\cap Y$ is a singleton set. For $1\le i\le m$, there are two cases. If $r_i\in B$, then each $p\in A$ has constant sign on $S_i$, so $S_i|_H$ is a singleton set. In this case $S_i\cap Y$ is also a singleton set. If $r_i\notin B$, then $r_i$ is a root of some unique $p \in A$. In this case $p$ changes sign at $r_i$ and the other elements of $A$ have constant sign on $S_i$, so $S_{i}|_H$ has two elements. In this case, $S_i\cap Y$ also has two elements. Also, different elements of $S_i\cap Y$ map to different elements of $S_i|_H$.

From the surjectivity of the map $Y \rightarrow X|_H$ it follows that the group homomorphism  $H \rightarrow G/Y^{\perp}$ is injective. Consequently, to complete the proof it suffices to establish the following:

Claim 2. $|H| = |G/Y^{\perp}|$. The elements of $\{ -1\} \cup A \cup \{ x-s_i : i=1,\dots,m+1\}$ form a $\mathbb{Z}/2\mathbb{Z}$-basis of $H$ so $|H| = 2^{|A|+m+2}$. Using (ii) we see that

\begin{align*}
|G/Y^{\perp}| = &\prod_{p\in A} |G/X_p^{\perp}|\cdot |G/X_{\infty}^{\perp}|\cdot \prod_{r\in B} |G/\| \sigma_r\}^{\perp}| \\ = &\prod_{p\in A} 2^{n_p+1}\cdot 2^2\cdot \prod_{r\in B} 2 = 2^{\sum_{p\in A} n_p+|A| +2+|B|}.
\end{align*}
At the same time, $m = \sum_{p\in A} n_p+|B|$, so $|A|+m+2 = \sum_{p\in A} n_p+|A| +2+|B|$.
\end{proof}

As an immediate consequence of Lemma \ref{key} we obtain a result of G\l adki and Jacob; see \cite[Theorem 1]{gj1}.

\begin{theorem} \label{gladki-jacob} The space of orderings $(X,G)$ %of $\mathbb{Q}(x)$
is profinite.
\end{theorem}

\begin{proof} %We continue to denote $(X_{\mathbb{Q}(x)}, G_{\mathbb{Q}(x)})$ by $(X,G)$, for short.
It suffices to show that for any finite subset $S$ of $G$ there exists a finite quotient $(X|_H,H)$ of $(X,G)$ such that $S\subseteq H$. Define $H$ as in the proof of Lemma \ref{key}, taking $A$ to be the set of real monic irreducible polynomials appearing in the factorization of the elements of $S$ and $B = \emptyset$. Then $H$ contains $S$ and $(X|_H,H)$ has the required properties.
\end{proof}

%Another consequence of Lemma \ref{key} is the following result.
%We continue to denote $(X_{\mathbb{Q}(x)}, G_{\mathbb{Q}(x)})$ by $(X,G)$, for short.
We mention another consequence of Lemma \ref{key}. Following the notation of Section \ref{general theory},
we fix a character $\gamma$ of $G$, $\gamma \ne 1$, $\gamma(-1)=1$, define $G_0 = \operatorname{ker}(\gamma)$, and $X_0 = X|_{G_0}$.

\begin{theorem}\label{2ndmain} The following are equivalent:
\begin{enumerate}
\item $(X_0,G_0)$ is a quotient space of $(X,G)$.
\item $\gamma$ satisfies the necessary conditions of Theorem \ref{nec cond}.
\item $(X_0,G_0)$ is a profinite space of orderings.
\end{enumerate}
\end{theorem}

We remark that \cite[Theorem 8]{gj2} asserts already that the implication (1) $\Rightarrow$ (3) of Theorem \ref{2ndmain} is true, but there are some gaps in the proof of \cite[Theorem 8]{gj2}.

\begin{proof} (1) $\Rightarrow$ (2) is a consequence of Theorem \ref{nec cond}. (3) $\Rightarrow$ (1) is trivial (since every profinite space of orderings is, in particular, a space of orderings). It remains to show (2) $\Rightarrow$ (3). Assume (2) holds. Let $\gamma = \prod_{i=1}^k \sigma_i$, $\sigma_i\in X$, $k= 2$ or $k=4$ and $\gamma \notin X^4$. To prove (3) it suffices to show that for any finite subset $S$ of $G_0$ there exists a finite quotient space $(X|_H,H)$ of $(X,G)$ such that $S \subseteq H$ and $(X|_{H\cap G_0},H\cap G_0)$ is a quotient space of $(X|_H,H)$. Define $H$, $Y$ as in the preamble to Lemma \ref{key}, taking $A$ to be any finite set of real monic irreducibles in $\mathbb{Q}[x]$ containing all real monic irreducible factors of elements of $S$ together with all real monic irreducibles $p$ such that $X_p \cap \{ \sigma_1,\dots,\sigma_k\} \ne \emptyset$, and taking $B$ to be any finite set of transcendental real numbers such that, for each $i=1,\dots,k$, if $\sigma_i$ is an archimedian ordering then the corresponding transcendental real number belongs to $B$. Obviously $S\subseteq H$. By Lemma \ref{key}, $(X|_H,H)$ is a quotient space of $(X,G)$ which is naturally identified with the subspace $(Y,G/Y^{\perp})$ of $(X,G)$. By construction, $Y$ is a union of connected components of $(X,G)$ and contains all components of $(X,G)$ meeting the set $\{ \sigma_1,\dots,\sigma_k\}$. Also, $(Y,G/Y^{\perp})$ is finite and has stability index $1$ or $2$. It follows, applying Theorem \ref{sap}, if the stability index is 1, or Theorem \ref{beyond sap}, if the stability index is 2, that $(Y|_{G_0}, G_0/Y^{\perp})$ is a quotient space of $(Y,G/Y^{\perp})$. Since $(X|_{H\cap G_0}, H\cap G_0)$ is identified with $(Y|_{G_0},G_0/Y^{\perp})$ under the isomorphism $(X|_H,H) \cong (Y,G/Y^{\perp})$, this completes the proof.
\end{proof}

\section{Examples of quotients of the space of orderings of the field $\Q(x)$}

Theorem \ref{2ndmain} provides us with an elegant criterion for checking whether a given quotient structure $(X_0, G_0)$ of $(X_{\Q(x)}, G_{\Q(x)})$ is a quotient space. In practice, however, there seems to be no good way of checking this criterion if $G_0$ is given in terms of generators, and we can, in fact, do this only in a few cases. We shall discuss this in some detail now.

Let $(X, G)$ be the space of orderings $(X_{\Q(x)}, G_{\Q(x)})$, and let $I$ denote the set of all (classes of) monic irreducible polynomials in $\Q[x]$ with at least one real root. Let $(X_0, G_0)$ be a fixed quotient structure of $(X, G)$ with $(G:G_0)=2$. Moreover, let $J \subseteq I$ be the set such that
$$G_0 = \langle \{-1\} \cup J \cup (I \setminus J)(I \setminus J) \rangle.$$
Observe also that $J = \{ p\in I : p \in G_0\}$, so $J$ determines uniquely and is uniquely determined by $G_0$.

\begin{example} If $J = I \setminus \{p\}$, for some $p \in I$, then $(X_0, G_0)$ is a quotient space. Indeed, suppose that $r \in \R$ is a root of $p$, and that $\sigma_r^-$ and $\sigma_r^+$ are the two orderings corresponding to $r$, one making $p$ positive, and one making $p$ negative. Let $\gamma = \sigma_r^- \cdot \sigma_r^+$. Then, readily, $G_0 = \ker \gamma$, and $(X_0, G_0)$ is a space of orderings by Theorem \ref{2ndmain}.
\end{example}

\begin{example} If $J = I \setminus \{p_1, p_2\}$, for some $p_1, p_2 \in I$, $p_1 \neq p_2$, then $(X_0, G_0)$ is a quotient space. As before, let $r_1, r_2 \in \R$ be real roots of $p_1, p_2$, respectively, and let  $\sigma_{r_i}^-$ and $\sigma_{r_i}^+$ be the two orderings corresponding to $r_i$, $i \in \{1, 2\}$. Let $\gamma = \sigma_{r_1}^- \sigma_{r_1}^+ \sigma_{r_2}^- \sigma_{r_2}^+$. Then, as before, $(X_0, G_0)$ is a quotient space by Theorem \ref{2ndmain} with $G_0 = \ker \gamma$.
\end{example}

\begin{example} If $J = I \setminus \{p_1, \ldots, p_n\}$, for some $n \geq 3$, and $p_1, \ldots, p_n \in I$ pairwise distinct, then $(X_0, G_0)$ is never a quotient space. Let $r_1, \ldots, r_n \in \R$ be real roots of $p_1, \ldots, p_n$, respectively, and let  $\sigma_{r_i}^-$ and $\sigma_{r_i}^+$ be the two orderings corresponding to $r_i$, $i \in \{1, \ldots, n\}$. Then $G_0 = \ker \sigma_{r_1}^- \sigma_{r_1}^+ \cdot \ldots \cdot \sigma_{r_n}^- \sigma_{r_n}^+$. Suppose that $(X_0, G_0)$ is a quotient space, and that $G_0 = \ker \gamma$, with $\gamma = \tau_1 \cdot \ldots \cdot \tau_4$, for some $\tau_1, \ldots, \tau_4 \in X$. Following an argument that will be later discussed in detail in the proof of Remark \ref{well-definedness} (2) we see, that the presentation $\sigma_{r_1}^- \sigma_{r_1}^+ \cdot \ldots \cdot \sigma_{r_n}^- \sigma_{r_n}^+$ cannot be shortened, and thus yield a contradiction.
%In particular
%$$\sigma_{r_1}^- \sigma_{r_1}^+ \cdot \ldots \cdot \sigma_{r_n}^- \sigma_{r_n}^+ = \tau_1 \cdot \ldots \cdot \tau_4.$$
%Let $Y_i$ be the connected component of $\tau_i$, $i \in \{1, \ldots, 4\}$. Then
%\begin{eqnarray*} \lefteqn{\tau_1 \cdot \ldots \cdot \tau_4, \sigma_{r_1}^- \sigma_{r_1}^+ \cdot \ldots \cdot \sigma_{r_n}^- \sigma_{r_n}^+ \in \chi(G|_{X_{p_1}} \oplus \ldots \oplus G|_{X_{p_n}} \oplus G|_{Y_1} \oplus \ldots \oplus  G|_{Y_4})} \\
%& = & \chi(G|_{X_{p_1}}) \oplus \ldots \oplus \chi(G|_{X_{p_n}}) \oplus \chi(G|_{Y_1}) \oplus \ldots \oplus \chi(G|_{Y_4})), \hspace{2cm} \end{eqnarray*}
%which implies that $Y_1, \ldots, Y_4 \in \{X_{p_1}, \ldots, X_{p_n}\}$. Say $\tau_1 \in X_{p_1}$ after reindexing, if necessary. Then also one of the remaining $\tau_i$, $i \in \{2, \ldots, 4\}$, belongs to $X_{p_1}$, for otherwise $\sigma_{r_1}^- \sigma_{r_1}^+ = \tau_1$ would be an ordering, which is impossible. Say $\tau_2 \in X_{p_1}$. Similarly $\tau_3, \tau_4 \in X_{p_{i_0}}$. Therefore $\tau_1 \tau_2 = \sigma_{r_1}^- \sigma_{r_1}^+$ and $\tau_3 \tau_4 = \sigma_{r_{i_0}}^- \sigma_{r_{i_0}}^+$, and consequently
%$$\iota = \frac{\sigma_{r_1}^- \sigma_{r_1}^+ \cdot \ldots \cdot \sigma_{r_n}^- \sigma_{r_n}^+}{\sigma_{r_1}^- \sigma_{r_1}^+\sigma_{r_{i_0}}^- \sigma_{r_{i_0}}^+} = 1.$$
%Since $n \geq 3$, we may choose $p \in \{p_2, \ldots, p_n\} \setminus \{p_{i_0}\}$. But then $\iota(p) = -1$, which leads to a contradiction.
\end{example}

\begin{example} If $J$ is finite, then $(X_0, G_0)$ is never a quotient space. For suppose $(X_0, G_0)$ is a quotient space with $G_0 = \ker \gamma$, for some $\gamma = \sigma_1 \cdot \ldots \cdot \sigma_4$, $\sigma_1, \ldots, \sigma_4 \in X$. Let $S$ be the finite set of all points on the real line corresponding to the orderings $\sigma_1, \ldots, \sigma_4$. Take an irreducible polynomial $q \in I$, strictly positive on the set $S$, but not belonging to $J$: we note that such a $q$ always exist, in fact, there are infinitely many such $q$. Then $q \in G_0$, which contradicts $q \notin J$. %Then $q \in \ker \gamma$, but clearly $q$ is neither divisible by any of the elements of $J$, nor a product of any two irreducilble polynomials -- a contradiction.
\end{example}

The case of both $J$ and $I \setminus J$ being infinite is widely open. %In principle, one would like to get a description of quotients similar to the ``positivity conditions'' of Example \ref{rofx}, but it is considerably more complicated compared to the case of the space of orderings of $\R(x)$.

\begin{example} Let $r_1, \ldots, r_4$ be the complete set of distinct real roots of an irreducible polynomial $q \in \mathbb{Q}[x]$. %four real algebraic numbers. Suppose they are all roots of one irreducible polynomial $q$.
Let
$$J = \{p \in I: p \mbox{ is positive at an even number of roots } r_i, i \in \{1, \ldots, 4\}\} \cup \{q\}.$$
Then $(X_0, G_0)$ is a quotient space. Indeed, one checks that if $\sigma_{r_i}^-$ and $\sigma_{r_i}^+$ are the two orderings corresponding to $r_i$, $i \in \{1, \ldots, 4\}$, then $G_0 = \ker \gamma$ for $\gamma = \sigma_{r_1}^- \sigma_{r_2}^- \sigma_{r_3}^- \sigma_{r_4}^-$, with $\sigma_{r_1}^-,  \sigma_{r_2}^-, \sigma_{r_3}^-, \sigma_{r_4}^-$ all coming from one connected component. We note that instead of $\sigma_{r_1}^- \sigma_{r_2}^- \sigma_{r_3}^- \sigma_{r_4}^-$ we can use any other combination of $\sigma_{r_i}^{\epsilon}$, $i \in \{1, \ldots, 4\}$, $\epsilon \in \{-,+\}$, making $q$ positive: at the end, they all define the same $\gamma$, since $\sigma_{r_i}^-\sigma_{r_i}^+\sigma_{r_j}^-\sigma_{r_j}^+=1$, for $i \neq j$, $i, j \in \{1, \ldots, 4\}$.
\end{example}

\begin{example} \label{tworoots} Let $r_1, \ldots, r_4$ be four real algebraic numbers again, but now suppose that $r_1$ and $r_2$ are roots of an irreducible polynomial $q_1$, and $r_3$, $r_4$ are roots of an irreducible $q_2$. Furthermore, assume that $q_1$ and $q_2$ have no roots other than $r_1, \ldots, r_4$. Let
$$J = \{p \in I: p \mbox{ is positive at an even number of roots } r_i, i \in \{1, \ldots, 4\}\} \cup \{q_1, q_2\}.$$
Then $(X_0, G_0)$ is a quotient space. Indeed, denote by $\sigma_{r_i}^-$ and $\sigma_{r_i}^+$ the two orderings corresponding to $r_i$, the first one making the minimal polynomial of $r_i$ negative, and the second one positive, $i \in \{1, \ldots, 4\}$. Depending on how the polynomials $q_1$ and $q_2$ overlap, there are different ways of defining $\gamma$. Say, for example, that $q_1$ and $q_2$ are related as in Figure \ref{fig_2}.
\begin{figure}[h!]
\centering%
\includegraphics[width=0.75\textwidth]{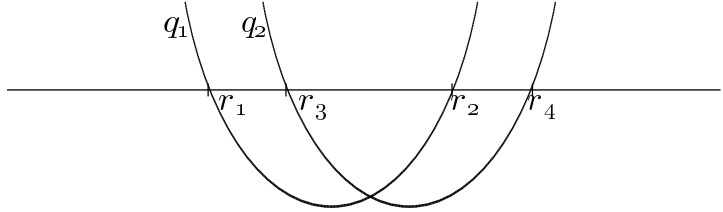}
\caption{Irreducible polynomials $q_1$ and $q_2$.}
\label{fig_2}
\end{figure}
Set $\gamma = \sigma_{r_1}^- \sigma_{r_3}^- \sigma_{r_2}^+ \sigma_{r_4}^+$. One checks that the conditions of Theorem \ref{2ndmain} are satisfied. The reader might wish to experiment with different ways of structuring roots of $q_1$ and $q_2$.
\end{example}

We note that if polynomials $q_1$ and $q_2$ in Example \ref{tworoots} have more than just two roots each, the quotient structure $(X_0, G_0)$ is, in general, not a quotient space. As details at this level are becoming too technical, we are not going to discuss this any further. %In a relatively similar way one can obtain ``positivity conditions'' for quotient structures associated to two real algebraic numbers, or four real transcendental numbers, or a number of a mix between algebraic and transcendental reals.

%We remark also that similar (but simpler) examples can be constructed starting with the space of orderings of $\mathbb{R}(x)$ instead of the space of orderings of $\mathbb{Q}(x)$.

\section{General quotients}

We continue to assume that $(X,G)$ is a space of orderings.
Fix a subgroup $G_0$ of $G$ containing $-1$, possibly having infinite index in $G$, and let $X_0$ denote the set of all restrictions of elements of $X$ to $G_0$. Denote the restriction of $\sigma \in X$ to $G_0$ by $\overline{\sigma}$. Let $S := X^4\cap \chi(G/G_0)$. Theorem \ref{main} generalizes as follows:

\begin{theorem} \label{main generalized}  A necessary condition for the quotient structure $(X_0,G_0)$ of $(X,G)$ to be a space of orderings is that $S$ generates $\chi(G/G_0)$ as a topological group, i.e., $\chi(G/G_0)$ is the closure of the subgroup of $\chi(G/G_0)$ generated by $S$, i.e., $S^{\perp} = G_0$.
\end{theorem}

\begin{proof} It suffices to show that, for each $g\in G\backslash G_0$ $\exists$ $\gamma \in S$ such that $\gamma(g) \ne 1$. Fix $g \in G\backslash G_0$. Case 1: $\exists$ $\sigma, \tau \in X$ such that $\sigma(g)\ne \tau(g)$ and $\overline{\sigma} = \overline{\tau}$. In this case we take $\gamma = \sigma\tau$. Case 2: $\forall$ $\sigma,\tau \in X$, $\overline{\sigma}=\overline{\tau}$ $\Rightarrow$ $\sigma(g)=\tau(g)$. In this case, the function $\phi : X_0 \rightarrow \{ \pm 1\}$ defined by $\phi(\overline{\sigma}) = \sigma(g)$ is well-defined and continuous, and is not in the image of the natural map $G_0 \hookrightarrow \operatorname{Cont}(X_0, \{ \pm 1\})$. Thus, by \cite[Theorem 7.2]{M}, $\exists$ a $4$-element fan $\overline{\sigma_1}, \overline{\sigma_2}, \overline{\sigma_3}, \overline{\sigma_4}$ in $X_0$ such that $\prod_{i=1}^4 \phi(\overline{\sigma_i}) \ne 1$. In this case we take $\gamma = \prod_{i=1}^4 \sigma_i$.
\end{proof}

For each $\gamma \in S$, $\gamma \ne 1$, $\gamma$ has some (not necessarily unique) minimal expression $\gamma = \prod_{i=1}^k \sigma_i$, $\sigma_i\in X$, $k = 2$ or $4$. Denote by $(Y,G/\Delta)$ the subspace of $(X,G)$ generated by the connected components of the various $\sigma_i$, $i=1,\dots,k$, $\gamma$ running through $S \backslash \{ 1 \}$, and let $Y_0$ denote the set of restrictions of elements of $Y$ to $G_0$.

\begin{theorem} \label{nec cond generalized} A necessary condition for the quotient structure $(X_0,G_0)$ of $(X,G)$ to be a space of orderings is that $S$ generates $\chi(G/G_0)$ as a topological group and the quotient structure $(Y_0,G_0/\Delta)$ of $(Y,G/\Delta)$ is a space of orderings.
\end{theorem}

\begin{proof} If the quotient structure $(X_0,G_0)$ of $(X,G)$ is a space of orderings then the quotient structure $(Y_0,G_0/\Delta)$ of $(Y,G/\Delta)$ is a subspace of $(X_0,G_0)$, so it is itself a space of orderings.
\end{proof}

The subspace $(Y,G/Y^{\perp})$ of $(X,G)$ defined above will be referred to as \it the core of the space of orderings $(X,G)$ with respect to the quotient structure $(X_0,G_0)$. \rm

\begin{rem} \label{well-definedness} (1) The connected components occurring in the definition of the core $(Y,G/\Delta)$ do not depend on the particular minimal presentations of the elements $\gamma \in S \backslash \{ 1 \}$. If we have two minimal presentations $$\gamma = \sigma_1\dots\sigma_k,\ \gamma = \tau_1\dots \tau_{\ell}, \ \sigma_i, \tau_j \in X, \ k,\ell \in \{2,4\}$$ then, using \cite[Lemma 3.2]{M-2}, we see that $\forall$ $i$ $\exists$ $j$ such that $\sigma_i \sim \tau_j$ and, similarly, $\forall$ $j$ $\exists$ $i$ such that $\tau_j \sim \sigma_i$.

\begin{proof} Since $\sigma_1\dots \sigma_k = \tau_1\dots \tau_{\ell}$ it follows that $\sigma_1\dots \sigma_k\tau_1\dots \tau_{\ell} = 1$. By hypothesis $\sigma_1,\dots,\sigma_k$ are linearly independent and $k\ge 1$. After reindexing suitably, we can assume that $\sigma_1,\dots,\sigma_k,\tau_1,\dots,\tau_{t-1}$, $1\le t \le \ell$ is a maximal linearly independent subset of $\sigma_1,\dots,\sigma_k,\tau_1,\dots,\tau_{\ell}$. Then $\tau_t$ is some linear combination of $\sigma_1,\dots,\sigma_k,\tau_1,\dots,\tau_{t-1}$, say $\tau_t = \prod_{i\in I} \sigma_i \prod_{j\in J} \tau_j$, $I \subseteq \{ 1,\dots,k\}$, $J \subseteq \{1,\dots, t-1\}$. Since $\tau_1,\dots,\tau_{\ell}$ are linearly independent we see that $I \ne \emptyset$. %and $\sigma_1\dots\sigma_k\tau_1\dots\tau_t = 1$ for some $1\le t \le \ell$.
According to \cite[Lemma 3.2]{M-2}, $\sigma_i \sim \tau_t$ for each $i \in I$. If $I= \{ 1,\dots.k\}$ we are done.  Otherwise, after canceling, we obtain $\prod_{i\in I'} \sigma_i \prod_{j\in J'} \tau_j = 1$ where $I' = \{ 1,\dots,k\} \backslash I$, $J' = \{ 1,\dots,\ell\} \backslash (J\cup \{ t\})$. The result follows now, by induction on $k$.
\end{proof}
(2) If $T$ is a maximal linearly independent subset of $S$ then the connected components coming from the elements of $S\backslash \{1\}$ are the same as the connected components coming from the elements of $T$.

\begin{proof} Suppose $\gamma \in S\backslash \{ 1 \}$, $\gamma = \gamma_1\dots \gamma_m$, $\gamma_j \in T$. Choose minimal presentations $\gamma = \prod_{i=1}^k \sigma_i$, $\gamma_j = \prod_{p=1}^{k_j} \tau_{jp}$, $\sigma_i, \tau_{jp} \in X$. We want to show that for each $i$, $\sigma_i \sim \tau_{jp}$ for some $j,p$. This reduces to showing $\sigma_1\dots \sigma_k = \tau_1\dots\tau_{\ell}$, $\sigma_i, \tau_j \in X$, $\sigma_1,\dots,\sigma_k$ linearly independent $\Rightarrow$ $\forall$ $i$ $\exists$ $j$ such that $\sigma_i \sim \tau_j$. Since it is possible to reduce further to the case where $\tau_1,\dots, \tau_{\ell}$ are linearly independent (by canceling whatever relations exist between the $\tau_j$, one by one) we see that this follows by the same argument used in (1).
\end{proof}
(3) Suppose $(X,G)$ has no infinite fans, $(G:G_0)= 2^m <\infty$, $S$ generates $\chi(G/G_0)$ as a (topological) group, and $\gamma_1,\dots,\gamma_m$ is some basis for $\chi(G/G_0)$ chosen so that each $\gamma_i$ belongs to $S$, and each $\gamma_i$ has a minimal presentation $\gamma_i = \prod_{j=1}^{k_i} \sigma_{ij}$, $\sigma_{ij}\in X$, then the core of the space of orderings $(X,G)$ with respect to the quotient structure $(X_0,G_0)$ is the union of the connected components of the various $\sigma_{ij}$. This follows from (2) in conjunction with the fact that any finite union of connected components is a subspace, by \cite[Theorem 3.6]{M2}.
\end{rem}

Again it is natural to wonder if the necessary conditions for a quotient structure to be a quotient space given by Theorem \ref{nec cond generalized} are sufficient. Although we are unable to prove this, we are able to show it is true in certain cases.

\begin{theorem} \label{SAP/finite} For a space of orderings $(X,G)$ with finitely many non-singleton connected components and no infinite fans  and a quotient structure $(X_0,G_0)$ of $(X,G)$ of finite index, the following are equivalent:
\begin{enumerate}
\item $(X_0,G_0)$ is a space of orderings.
\item $X^4\cap \chi(G/G_0)$ generates $\chi(G/G_0)$ and the quotient structure $(Y_0,G_0/\Delta)$ of the core $(Y,G/\Delta)$ is a space of orderings.
\end{enumerate}
\end{theorem}

\begin{proof} Apply Corollary \ref{sheafcor}.
\end{proof}

Observe that Theorem \ref{SAP/finite} applies, in particular, to finite spaces of orderings and to SAP spaces of orderings.

\begin{theorem} \label{Q(x)} For the space of orderings $(X,G) = (X_{\mathbb{Q}(x)}, G_{\mathbb{Q}(x)})$ and a quotient structure $(X_0,G_0)$ of $(X,G)$ of finite index, the following are equivalent:
\begin{enumerate}
\item $(X_0,G_0)$ is a space of orderings.
\item $X^4\cap \chi(G/G_0)$ generates $\chi(G/G_0)$ and the quotient structure $(Y_0,G/\Delta)$ of the core $(Y,G/\Delta)$ is a space of orderings.
\item $(X_0,G_0)$ is a profinite space of orderings.
\end{enumerate}
\end{theorem}

\begin{proof} (1) $\Rightarrow$ (2) is a consequence of Theorem \ref{nec cond generalized}. (3) $\Rightarrow$ (1) is trivial (since every profinite space of orderings is, in particular, a space of orderings). It remains to show (2) $\Rightarrow$ (3). Assume (2) holds. Let $(G:G_0) = 2^m$, and choose $\gamma_i = \prod_{j=1}^{k_i}\sigma_{ij}$, $\sigma_{ij}\in X$, $i = 1,\dots,m$ as in Remark \ref{well-definedness}(3). To prove (3) it suffices to show that for any finite subset $W$ of $G_0$ there exists a finite quotient space $(X|_H,H)$ of $(X,G)$ such that $W \subseteq H$ and $(X|_{H\cap G_0},H\cap G_0)$ is a quotient space of $(X|_H,H)$. Define $H$, $Y$ as in the preamble to Lemma \ref{key}, taking $A$ to be any finite set of real monic irreducibles in $\mathbb{Q}[x]$ containing all real monic irreducible factors of elements of $W$ together with all real monic irreducibles $p$ such that $X_p \cap \cup\{ \sigma_{ij}: i=1,\dots,m, j = 1,\dots,k_i\} \ne \emptyset$, and taking $B$ to be any finite set of transcendental real numbers such that, for each $i=1,\dots,m$ and each $j=1,\dots, k_i$, if $\sigma_{ij}$ is an archimedian ordering then the corresponding transcendental real number belongs to $B$. Obviously $W\subseteq H$. By Lemma \ref{key}, $(X|_H,H)$ is a quotient space of $(X,G)$ which is naturally identified with the quotient space $(Y,G/Y^{\perp})$ of $(X,G)$. By construction, $Y$ is a union of connected components of $(X,G)$ and contains all components of $(X,G)$ meeting the set $\{ \sigma_{ij}: i=1,\dots,m, j=1,\dots,k_i\}$. Also, $(Y,G/Y^{\perp})$ is finite and has stability index $1$ or $2$. It follows from Theorem \ref{SAP/finite} applied to the finite space of orderings $(Y,G/Y^{\perp})$ and assumption (2) that $(Y|_{G_0}, G_0/Y^{\perp})$ is a quotient space of $(Y,G/Y^{\perp})$. Since $(X|_{H\cap G_0}, H\cap G_0)$ is identified with $(Y|_{G_0},G_0/Y^{\perp})$ under the isomorphism $(X|_H,H) \cong (Y,G/Y^{\perp})$, this completes the proof.
\end{proof}

\begin{rem} In \cite[Proposition 6]{AstTre} and \cite[Theorem 2]{gj1} it has been shown that the Lam's Open Problem B holds true for any profinite spaces of orderings. Thus Theorem \ref{Q(x)} shows, in particular, that a conceivable method of finding non-realizable spaces of orderings among quotients of the space of orderings of $\Q(x)$ of finite index will prove to be fruitless. At the same time, it remains an open problem whether profinite spaces of orderings are realizable. We note here that the dual question of whether direct limits of finite spaces of orderings are realizable was partially answered already in the early 1980's in \cite{KMS}, and recently completely resolved in \cite{AM}
\end{rem}

\section{Appendix: The sheaf construction}\label{background}

We recall %(a special case of)
the sheaf construction in \cite[Chapter 8]{M1}. The results in \cite[Chapter 8]{M1} are phrased in terms of reduced Witt rings, not spaces of orderings, but the two categories are equivalent, so these results are valid for spaces of orderings.

\begin{theorem}\label{sheaf} Suppose $(X_i,G_i)$ is a space of orderings for each $i \in I$, where $I$ is a Boolean space. Suppose $X = \dot{\cup}_{i\in I} X_i$ is equipped with a topology such that
\begin{enumerate}
\item $X$ is a Boolean space,

\item the inclusion map $X_i \hookrightarrow X$ is continuous, for each $i \in I$,

\item the projection map $\pi : X \rightarrow I$ is continuous, and

\item if $(i_{\lambda})_{\lambda \in D}$ is any net in $I$ converging to $i\in I$ and if $\sigma_1^{\lambda}, \sigma_2^{\lambda}, \sigma_3^{\lambda}, \sigma_4^{\lambda}$ is a 4-element fan in $X_{i_{\lambda}}$ such that $\sigma_j^{\lambda}$ converges to $\sigma_j \in X_i$ for each $j=1,2,3,4$, then $\sigma_1\sigma_2\sigma_3\sigma_4 = 1$.
\end{enumerate}
Then $(X,G)$ is a space of orderings, where
$$G := \{ \phi \in \operatorname{Cont}(X, \{ \pm 1\}):  \phi|_{X_i} \in \hat{G_i} \, \forall \, i\in I \}.$$
\end{theorem}

\begin{proof} See \cite[Theorem 8.5]{M1}.
\end{proof}
%See %\cite[conditions (i)-(iv), page 200]{M1} and
% for the proof. We do not need this stronger version here.

We need only the following special case of Theorem \ref{sheaf}:

\begin{cor}\label{sheafcor} Suppose $(X_i,G_i)$ is a space of orderings for each $i \in I$, where $I$ is a Boolean space, and $(X_i,G_i)$ is SAP for all but finitely many $i$. Suppose $X = \dot{\cup}_{i\in I} X_i$ is equipped with a topology such that $X$ is a Boolean space, the inclusion map $X_i \hookrightarrow X$ is continuous, for each $i \in I$, and the projection map $\pi : X \rightarrow I$ is continuous. Then $(X,G)$ is a space of orderings, where $$G := \{ \phi \in \operatorname{Cont}(X, \{ \pm 1\}):  \phi|_{X_i} \in \hat{G_i} \, \forall \, i\in I \}.$$  %, i.e., $|X_i| = 1$, $G_i = \{ \pm 1\}$.
\end{cor}

\begin{proof} It suffices to show that condition (4) of Theorem \ref{sheaf} holds. Suppose $(i_{\lambda})_{\lambda\in D}$ is a net in $I$ satisfying the hypothesis of (4). For each $\lambda \in D$, $X_{i_{\lambda}}$ contains a 4-element fan (so, in particular, the space of orderings $(X_{i_{\lambda}}, G_{i_{\lambda}})$ is not SAP) so the set $\{ i_{\lambda}: \lambda \in D\}$ is finite. Replacing the net $(i_{\lambda})_{\lambda\in D}$ by a suitable subnet, we can assume the net $(i_{\lambda})_{\lambda \in D}$ is constant. In this case, the conclusion of (iv) is obvious, using the continuity of the multiplication in the character group $\chi(G_i)$.
\end{proof}

\end{document}